\documentclass[12pt,a4paper, reqno]{amsart}
\usepackage{amsmath,amsfonts,amssymb,amsthm,amscd}

\input{comment.sty}
\includecomment{MM}
\includecomment{CL}
\usepackage[dvipdfmx]{graphicx}
\numberwithin{equation}{section}

\newcommand{\E}{\mathbb{E}}
\newcommand{\Pb}{\mathbb{P}}
\newcommand{\R}{\mathbb{R}}
\newcommand{\Z}{\mathbb{Z}}
\newcommand{\id}{$id$}

\renewcommand{\P}{\mathcal{P}}
\newcommand{\A}{\mathcal{A}}

\newcommand{\G}{\Gamma}
\renewcommand{\d}{\delta}

\renewcommand{\i}{\iota}

\renewcommand{\l}{\lambda}
\renewcommand{\L}{\Lambda}
\newcommand{\m}{\mu}
\newcommand{\n}{\nu}

\renewcommand{\r}{\rho}
\newcommand{\s}{\sigma}

\newcommand{\e}{\varepsilon}
\newcommand{\f}{\varphi}

\newcommand{\y}{\eta}

\newtheorem{thm}{Theorem}[section]
\newtheorem{lem}{Lemma}[section]

\newtheorem{Def}{Definition}[section]

\newtheorem{rem}{Remark}[section]
\newtheorem{ex}{Example}[section]

\begin{document}

\title{A note on a local ergodic theorem for an infinite tower of coverings.}
\author{Ryokichi Tanaka}
\address{Advanced Institute for Materials Research and Mathematical Institute, Tohoku University, 2-1-1 Katahira, Aoba-ku, Sendai, 980-8577, Japan}
\email{rtanaka@wpi-aimr.tohoku.ac.jp}
\date{}

\maketitle

\begin{abstract}
This is a note on a local ergodic theorem for a symmetric exclusion process defined on an infinite tower of coverings, which is associated with a finitely generated residually finite amenable group.
\end{abstract}

\section{Introduction}

The {\it hydrodynamic limit} is one of the main frameworks to understand scaling limit of interacting particle systems in order to capture the relation between microscopic and macroscopic phenomena in statistical physics.
There have been studied a number of models such as exclusion processes (e.g., \cite{FHU} and \cite{KLO}), the stochastic Ginzburg-Landau model (e.g., \cite{GPV}), a chain of anharmonic oscillators \cite{OS}
and  stochastic energy exchange models \cite{Sa}.
See also \cite{KL} and \cite{S} for a background of this problem and its history.
While a large number of studies have been devoted to understand stochastic models on the one-dimensional or the $d$-dimensional discrete torus and their scaling limits, several attempts have been made to generalise the underlying graphs to the ones with rich geometric structure.
For example, Jara has studied a zero-range process on the Sierpinski gasket \cite{J} and obtained a nonlinear heat equation which the limit of density empirical measure satisfies.
The author gave a unified framework to discuss exclusion processes on a general class of graphs and obtained the hydrodynamic limit result for crystal lattices \cite{T}.
Recently, Sasada has given a unified and simplified approach to obtain the hydrodynamic limit for an important class of exclusion processes, so-called of {\it non-gradient} type, in general settings \cite{SK}.
Although it is tempting to study a number of stochastic models on general underlying graphs from viewpoint of physics, 
the attempts made so far depend on the models and also on the structure of graphs.
This note proposes a possible approach to unify these generalisations on crystal lattices and self-similar graphs, and gives a {\it local ergodic theorem} for exclusion processes for a step toward to obtain the scaling limit.

In group theory, there is a remarkable class of groups which have been studied in the context of dynamical systems.
Bartholdi, Grigorchuk and Nekrashevych described a number of self-similar (fractal) sets as the scaling limits of finite graphs associated with group actions \cite{BGN}.
They showed that a class of Julia sets and the Sierpinski gasket can be obtained as the scaling limits of {\it Schreier} graphs of some groups.
See also \cite{K} for another (yet related) construction of such a limit.
These examples appear as groups acting on rooted trees, and a sequence of Schreier graphs is associated with the action on tree (\cite{N} for a general background on this topic).
Here we discuss the simplest case, an {\it infinite tower of coverings} associated with those kind of groups (Section \ref{tower}).
The group we consider is {\it finitely generated, infinite, residually finite amenable} (see Section \ref{tower} for the definition).
We define an exclusion process on each finite covering graph, and take a limit as the size of graphs goes to infinity.

A limit theorem we are considering is a local ergodic theorem which describes local equilibrium states of particle systems.
The theorem is also referred as the {\it replacement lemma}, e.g., in \cite{KL} and \cite{T}, and enables us to replace a local average of number of particles by a global average after a large enough time.
There involves an ergodic theoretic argument to exchange the space average and the time average.
In the hydrodynamic limit, this theorem has a role to describe several different local equilibrium states, and to verify the derivation of a (possibly non-linear) partial differential equation by pasting together those states.
Here we formulate a local ergodic theorem by using a notion of {\it local function bundle} introduced in \cite{T}, and show it in the form of super exponential estimate (Theorem \ref{superexponential}).
The proof is based on the entropy method due to \cite{GPV} and also \cite{KOV}.
In order to obtain the scaling limit result, one has to describe a limiting space, which we will not discuss in this note.
The problem we discuss here is the first step before taking the scaling limit.
It would be interesting to complete the second step and to obtain the hydrodynamic limit.

The organisation of this note is the following.
In Section \ref{result}, we introduce the setting and examples of infinite tower of coverings, define exclusion processes and formulate a local ergodic theorem (Theorem \ref{superexponential}).
In Section \ref{proof}, we show Theorem \ref{superexponential} and also prove auxiliary results: the one-block estimate (Theorem \ref{1block}) and the two-blocks estimate (Theorem \ref{2blocks}).

\section{Notation and results}\label{result}

\subsection{Groups and associated infinite towers of coverings}\label{tower}

Let $\Gamma=\langle S\rangle$ be a finitely generated group, where $S$ is a finite symmetric set of generators, i.e., $S=S^{-1}$.
Throughout this note $\Gamma$ is an infinite group.
We assume that $\Gamma$ is {\it residually finite}, i.e.,
there is a descending sequence of finite index normal subgroups $\{\Gamma_i\}_{i=0}^\infty$ such that  
$\Gamma_{i}\lhd \Gamma$ with $[\Gamma:\Gamma_{i}] < \infty$ for each $i$, 
and $\bigcap_{i=1}^{\infty} \Gamma_{i}=\{\id\}$.
We also assume that $\Gamma$ is {\it amenable}, i.e., there exists a sequence of finite subsets $\{F_{i}\}_{i=0}^\infty$ of $\Gamma$ such that $\lim_{i \to \infty}|\partial_{S}F_{i}|/|F_{i}|=0$ where $\partial_{S}F_{i}:=F_{i}S\setminus F_{i}$.
We call such a sequence $\{F_i\}_{i=0}^\infty$ a {\it F{\o}lner sequence} and each $F_i$ a {\it F{\o}lner set}.
Henceforth we fix some finite generating set $S$ in $\G$ and denote by $|x|_{\G}$ the corresponding word norm of $x$ in $\G$, i.e., the minimum number of elements in $S$ to obtain $x$.

\begin{ex}\upshape
\begin{itemize} \ \
\item[(i)]{\it The group of integers $\Z$.}
We can take $S=\{-1, 1\}$ as a finite symmetric set of generators.
Considering the subgroups $\Gamma_{i}:=2^{i}\Z$ for positive integers $i$, one can check that the group is residually finite.
The sets $F_{i}:=[-i,i] \subset \Z$ gives a F{\o}lner sequence.
In the same way, $\Z^{d} (d \ge 1)$ is also residually finite and amenable.
\item[(ii)]{\it Grigorchuck group.}
Grigorchuck group is residually finite and a subsequence of balls forms a F{\o}lner sequence since it has subexponential volume growth. (See, e.g., Chapter 6 in \cite{CSC}.)
\item[(iii)]{\it Basilica group.}
Basilica group is realized as a finitely generated subgroup in the automorphism group of the binary tree, and this implies that the group is residually finite. The group is also amenable as proved in \cite{BV}.
\end{itemize}
\end{ex}

Let $X=(V, E)$ be a Cayley graph of $\G$ associated with $S$, i.e., 
the set of vertices $V$ is $\G$, and the set of oriented edges $E:=\{ (x, xs) \in \G \times \G \ : \ s \in S\}$.
Here we consider the Cayley graph as an oriented graph where both possibilities of orientation are included, i.e., if $e \in E$, then its reversed edge $\overline e \in E$. 
Denote the origin of $e$ by $oe$ and the terminus by $te$.
The group $\Gamma$ acts on $X$ from the left hand side freely and vertex transitively. 
The quotient graph $X_{0}:=\Gamma\backslash X$ consists of one vertex and (oriented) loop edges.
We use the graph distance $d$ in $X$, and denote by $B(x, r)$ the ball in $V$ centred at $x$ with radius $r$.
Abusing the notation, we denote the graph distance by $d$ in other graphs as well.

\begin{rem}\upshape
We can extend our results to quasi-transitive graphs, i.e., $\Gamma$ acts on $X$ with a finite number of orbits. 
If $\Gamma=\Z^{d}$, the quasi-transitive graph $X$ is called a {\it crystal lattice} (\cite{KS01} and \cite{T}).
\end{rem}

Let $\{\Gamma_i\}_{i=0}^\infty$
be a descending sequence of normal subgroups such that $\Gamma_{i}\lhd \Gamma$ and $[\Gamma:\Gamma_{i}] < \infty$ for each $i$, and $\bigcap_{i=1}^{\infty} \Gamma_{i}=\{\id\}$.
For each $i$, define the quotient finite graph $X_{i}:=\Gamma_{i}\backslash X$. 
Each $X_{i}$ is a finite graph since $\Gamma_{i}$ is a finite index subgroup of $\Gamma$.
Here $\Gamma/\Gamma_{i}$ acts on $X_{i}$ freely and vertex transitively, where $\Gamma/\Gamma_{i}$ is a finite group.
Then we have an infinite tower of coverings of finite graphs: $X_{i} \to X_{0}$ for $i=1, 2, \dots$.

Let us define a particle system on $X$.
Define the configuration space by 
$Z:=\{0,1\}^{V}$, and denote each configuration by $\eta:=\{\eta_{x}\}_{x \in V}$.
The action of $\G$ lifts on $Z$ naturally, by setting $(\s\y)_{z}:=\y_{\s^{-1}z}$ for $\s \in \G$, $\y \in Z$ and $z \in V$.
In the same way, for each quotient finite graph $X_{m}=(V_{m}, E_{m})=\G_{m}\backslash X$, we define a configuration space $Z_{m}:=\{0,1\}^{V_{m}}$.
The action of $\G/\G_{m}$ on $X_{m}$ lifts on $Z_{m}$ as above.
Here we define a {\it local function bundle}, which is defined on the product space of the state space $V$ and the configuration space $Z$.
This is used to formulate our local ergodic theorem.

\begin{Def}
A $\G$-invariant local function bundle for vertices is a function $f: V \times Z \to \R$ such that:
\begin{itemize}
\item There exists $r \ge 0$ such that for every $x \in V$, the function $f_{x}:Z \to \R$ depends only on $\{\y_{z}\}_{z \in B(x, r)}$. 
\item For every $\s \in \G, x \in V, \y \in Z$, it holds that $f(\s x, \s \y)=f(x,\y)$.
\end{itemize}
In a similar way, a $\G$-invariant local function bundle for edges is a function $f:E \times Z \to \R$ such that:
\begin{itemize}
\item There exists $r \ge 0$ such that for every $e \in E$, the function $f_{e}: Z \to \R$ depends only on $\{\y_{z}\}_{z \in B(oe, r)\cup B(te, r)}$.
\item For all $\s \in \G, e \in E, \y \in Z$, it holds that $f(\s e, \s \y)=f(e, \y)$.
\end{itemize}
\end{Def}

\begin{ex}\upshape \ \
\begin{itemize}
\item[(i)] For $x \in V$, if we define $f_{x}(\y):=\y_{x}$, then $f$ is a $\G$-invariant local function bundle for vertices.
In this case, for each $x \in V$, the function $f_{x}$ depends only on a configuration on $x$. 
The $f$ is $\G$-invariant by definition.
\item[(ii)] For $x \in V$, if we define $f_{x}(\y):=\prod_{e \in E_{x}}\y_{te}$, where $E_{x}:=\{e \in E \ : \ oe=x\}$, then $f$ is a $\G$-invariant local function bundle for vertices.
\item[(iii)] For $e \in E$, if we define $f_{e}(\y):=\y_{oe}+\y_{te}$, then $f$ is a $\G$-invariant local function bundle for edges. 
\item[(iv)] For $e \in E$, if we define $f_{e}(\y):=\prod_{e' \in E_{oe}}\y_{te'}+\prod_{e' \in E_{te}}\y_{te'}+c$, where $c>0$, then $f$ is a $\G$-invariant local function bundle for edges and satisfies $f(e, \y)\ge c$ for all $e \in E, \y \in Z$.
\end{itemize}
\end{ex}

Let $\{F_i\}_{i=1}^{\infty}$ be a F{\o}lner sequence for $\G$.
For a $\G$-invariant local function bundle for vertices $f:V \times Z \to \R$,
we define a {\it local average} associated with $\{F_{i}\}_{i=1}^{\infty}$.
For $x \in V$ and $F_{i}$, let
$$
\overline f_{x, i}:=\frac{1}{|F_{i}|}\sum_{\s \in F_{i}}f_{\s x}.
$$
Note that $\overline f_{\cdot, i}$ is again a $\G$-invariant local function bundle.

Let us introduce a local average on a F{\o}lner set controlled by its size.
For a non negative real number $K$, we define
$$
b(K):=\max \{i \ : \ F_{i} \subseteq B(o, K)\},
$$
i.e.,
$b(K)$ is the largest number $L$ such that all $F_{i}, i=1, \dots, L$ are included in the ball centred at $o$ with radius $K$.
Note that $b(K)$ is non-decreasing and goes to $\infty$ as $K$ goes to $\infty$.
We also consider a local average of the following type:
$$
\overline f_{x, b(K)}:=\frac{1}{|F_{b(K)}|}\sum_{\s \in F_{b(K)}}f_{\s x}.
$$

Fix a distinguished vertex $o \in V$.
Now $\G$ acts on $V$ transitively, so every vertex $x$ can be written in the form $x=\gamma o$ for some $\gamma \in \G$.
For every $m \ge 1$, denote by the same character $o \in V_{m}$, the image of $o$ via the covering map $X \to X_{m}$.
For $\s \in \G$, denote by $\underline \s \in \G\slash\G_{m}$ the image of $\s$ via the canonical surjection.
Since $f$ is $\G$-invariant, the $\G$-invariant local function bundles $f$ for vertices and edges induce functions on $V_{m}\times Z_{m}$, and $E_{m}\times Z_{m}$ respectively, they are $\G/\G_{m}$-invariant under the diagonal action for each $m$.
We also use the same character for these induced ones.

\subsection{Particle systems}

Assume that $Z_{m}$ and $Z$ are equipped with the prodiscrete topology, i.e., the product of discrete topology.
Denote by $\P(Z_{m})$ and by $\P(Z)$ the spaces of Borel probability measures on $Z_{m}$ and on $Z$, respectively.
We also define the $(1/2)$-Bernoulli measures $\n_{m}$ on $Z_{m}$ and $\n$ on $Z$, as the direct product of the $(1/2)$-Bernoulli measures on $\{0,1\}$.

Let us define the symmetric exclusion process on $X$. For a configuration $\y \in Z$ and an edge $e \in E$, define by $\y^{e}$ a configuration exchanging states on $oe$ and $te$, i.e.,
$$
\y^{e}_{z}:=
\begin{cases}
\y_{te} & z=oe, \\
\y_{oe} & z=te, \\
\y_{z} & \text{otherwise},
\end{cases}
$$
for $z \in V$.
Note that $\y^e=\y^{\overline{e}}$.
Furthermore, for a function $F$ on $Z$ and $e \in E$, we define $\pi_{e}F(\y):=F(\y^{e}) - F(\y)$.
We also define the corresponding ones for $\y \in Z_{m}$ and $e \in E_{m}$ in the same way.
The {\it symmetric exclusion process} on $X$ is defined in terms of a $\G$-invariant local function bundle for edges.
We call a $\G$-invariant local function bundle $c: E \times Z \to \R$ is a {\it jump rate} if it satisfies the following properties:
\begin{itemize}
\item (symmetric) $c(e, \y)=c(\overline e, \y)$ and $c(e, \y)=c(e, \y^{e})$ for all $e \in E, \y \in Z$.
\item (non-degenerate) There exists a positive constant $c_{0}$ such that $c_{0} \le c(e, \y)$ for all $e \in E, \y \in Z$.
\end{itemize}
Define the generator of a symmetric exclusion process by the operator
$L_{m}: L^{2}(Z_{m}, \n_{m}) \to L^{2}(Z_{m}, \n_{m})$,
$$
L_{m}F(\y):=\sum_{e \in E_{m}}c(e, \y)\pi_{e}F(\y),
$$
for $F \in L^{2}(Z_{m}, \n_{m})$.

We consider a family of exclusion processes on an infinite tower of coverings.
Each process on a covering graph $X_{m}$ is speeded up by some time scaling factor $t_{m}$.
Suppose that $\{t_{m}\}_{m=1}^{\infty}$ is an increasing sequence of positive numbers:
$t_{1} < t_{2} < \cdots < t_{m} \to \infty$ with the condition $\sqrt{t_{m}} \le 2 diam X_{m}$,
where $diam X_{m}$ denotes the diameter of $X_{m}$.

Consider the continuous time Markov chain generated by $L_{m}$ speeded up by $t_{m}$.
Fix an arbitrary time $T >0$ and denote by $D([0,T], Z_{m})$ the space of paths being right continuous and having left limits.
For a probability measure $\m_{m}$ on $Z_{m}$, 
we define by $\Pb_{m}$ the distribution on $D([0,T], Z_{m})$ of the continuous time Markov chain $\y_{m}(t)$ generated by $t_{m}L_{m}$ with the initial measure $\m_{m}$.

For $0 \le \r \le 1$, denote by $\n_{\r}$ the $\r$-Bernoulli measure on $Z$, that is the direct product of the $\r$-Bernoulli measures on $\{0,1\}$.
We define a {\it global average} of a $\G$-invariant local function bundle 
$f:V \times Z \to \R$ by $\langle f_{o}\rangle(\r):=\E_{\n_{\r}}[f_{o}]$
the expectation of the function $f_{o}: Z \to \R$ with respect to $\n_{\r}$.

The following estimate enables us to approximate a global average of a local function bundle by a local average of one under a suitable time-space average.
This estimate is referred as a local ergodic theorem, which we show in the super exponential estimate.

\begin{thm}\label{superexponential}
Fix $T >0$. For every $\G$-invariant local function bundles $f:V\times Z \to \R$ and every $\d >0$,
$$
\lim_{i \to \infty}\limsup_{\e \to 0}\limsup_{m \to \infty}\frac{1}{[\G:\G_{m}]}\log \Pb_{m}\left(\frac{1}{[\G:\G_{m}]}\int_{0}^{T}V_{o, m, \e, i}(\y(t))dt \ge \d \right)= -\infty,
$$
where 
$$
V_{o, m, \e, i}(\y)=\sum_{\underline \s \in \G/\G_{m}}\left|\overline f_{\underline \s o, i}(\y) - \langle f_{o}\rangle \left(\overline \y_{\underline \s o, b(\e\sqrt{t_{m}})}\right)\right|.
$$
\end{thm}

Here we remark that this theorem generalises \cite{T}[Theorem 4.1] for crystal lattices.

\section{Proof of Theorem \ref{superexponential}}\label{proof}

The super exponential estimate for $\Pb_{m}$ is reduced to $\Pb_{m}^{eq}$ which is the distribution of continuous time Markov chain generated by $t_{m}L_{m}$ with the initial measure the (1/2)-Bernoulli measure $\n_{m}$, i.e., an equilibrium measure.
Indeed, for every measurable set $A \subset D([0,T], Z_{m})$,
$\Pb_{m}(A) \le 2^{[\G:\G_{m}]}\Pb_{m}^{eq}(A)$ and the factor $2^{[\G:\G_{m}]}$ does not contribute to the super exponential estimate.
(Note that $[\G:\G_m]=|V_m|$.)
Moreover, the super exponential estimate Theorem \ref{superexponential} is reduced to the following theorem.

For $\m \in \P(Z_{m})$, define the Dirichlet form for $\f:=d\m/d\n_{m}$ by 
$$
I_{m}(\m):=-\int_{Z_{m}}\sqrt{\f}L_{m}\sqrt{\f}d\n_{m},
$$
which is also equal to $(1/2)\int_{Z_{m}}\sum_{e \in E_{m}}c(e, \y)\left(\pi_{e} \sqrt{\f}\right)^{2}d\n_{m}$.
For every $C>0$, we define the subset of $\P(Z_{m})$ by 
$$
\P_{m,C}:=\left\{ \mu \in \P(Z_{m}) \ : \ \text{$\m$ is $\G/\G_{m}$-invariant and} \ I_{m}(\m)\le C\frac{[\G:\G_{m}]}{t_{m}} \right\}.
$$

\begin{thm}\label{equilibrium}
For every $C>0$, it holds that
$$
\lim_{i \to \infty}\limsup_{\e \to 0}\limsup_{m \to \infty}\sup_{\m \in \P_{m,C}}\E_{\m}\left|\overline f_{o, i} - \langle f_{o}\rangle(\overline \y_{o, b(\e\sqrt{t_{m}})})\right|=0.
$$
\end{thm}

First, we see how to deduce Theorem \ref{superexponential} from Theorem \ref{equilibrium}.

\begin{proof}[Proof of Theorem \ref{superexponential}]
As we observe, it suffices to prove the required estimate for $\Pb_{m}^{eq}$.
Recall that
$V_{o, m, \e, i}(\y)=\sum_{\underline \s \in \G/\G_{m}}\left|\overline f_{\underline \s o, i}(\y) - \langle f_{o}\rangle \left(\overline \y_{\underline \s o, b(\e\sqrt{t_{m}})}\right)\right|$.

By the Chebychev inequality, for all $a >0$ and for all $\d >0$,
$$
\Pb_{m}^{eq}\left(\frac{1}{[\G:\G_{m}]}\int_{0}^{T}V_{o, m, \e, i}(\y(t))dt \ge \d \right)
\le \E_{m}^{eq} \exp\left(a\int_{0}^{T}V_{o, m, \e, i}(\y(t))dt -a \d [\G:\G_{m}] \right).
$$

For all $a >0$, we consider the operator 
$$
t_{m}L_{m} + a V_{o, m, \e, i} : L^{2}(Z_{m}, \n_{m}) \to L^{2}(Z_{m}, \n_{m}),
$$
which is self-adjoint for all $a>0$ by the definition of $L_{m}$.
Denote by $\l_{o, m, \e, i}(a)$ the largest eigenvalue of $t_{m}L_{m} + aV_{o, m, \e, i}$.
By the Feynmann-Kac formula (e.g., \cite{KL}[Lemma 7.2, Appendix 1]),
$$
\E_{m}^{eq}\exp\left(a\int_{0}^{T}V_{o, m, \e, i}dt\right) \le \exp T\l_{o, m, \e, i}(a).
$$
Therefore it suffices to show that for all $a >0$,
\begin{equation}\label{eigenvalue}
\lim_{i \to \infty}\limsup_{\e \to 0}\limsup_{m \to \infty}\frac{1}{[\G:\G_{m}]}\l_{o, m, \e, i}(a)=0. 
\end{equation}
Indeed, by using (\ref{eigenvalue}), we have that
$$
\lim_{i \to \infty}\limsup_{\e \to 0}\limsup_{m \to \infty}\frac{1}{[\G:\G_{m}]}\log \Pb_{m}^{eq}\left(\frac{1}{[\G:\G_{m}]}\int_{0}^{T}V_{o, m, \e, i}(\y(t))dt \ge \d \right)\le -a\d.
$$
Taking $a \to \infty$, we obtain Theorem \ref{superexponential}.

It remains to prove (\ref{eigenvalue}).
By the variational principle,
the largest eigenvalue $\l_{o, m, \e, i}(a)$ can be expressed in the following form:
$$
\l_{o, m, \e, i}(a):=\sup_{\m \in \P(Z_{m})}\left\{a \int_{Z_{m}}V_{o, m, \e, i} d\m - t_{m}I_{m}(\m)\right\}.
$$
It is enough to consider only the case when 
$a \int_{Z_{m}}V_{o, m, \e, i} d\m \ge t_{m}I_{m}(\m)$ for some $\m \in \P(Z_m)$.
For $\m \in \P(Z_{m})$, we deduce by $\overline \m$ the average of $\m$ by the $\G/\G_{m}$-action, that is,
$$
\overline \m:=\frac{1}{[\G:\G_{m}]}\sum_{\underline \s \in \G/\G_{m}}\m\circ \underline \s.
$$
Here $\overline \m$ is a $\G/\G_{m}$-invariant probability measure.
Now we have that
$$
\frac{1}{[\G:\G_{m}]}\int_{Z_{m}}V_{o, m, \e, i}d\m =\E_{\overline \m}\left|\overline f_{o, i} - \langle f_{o}\rangle (\overline \y_{o, b(\e\sqrt{t_{m}})})\right|.
$$
There exists a constant $C(f)>0$ depending only on $f$ such that
$$
V_{o, m, \e, i} \le C(f)[\G:\G_{m}].
$$
By the convexity of $I_{m}$,
$$
I_{m}(\overline \m)\le \frac{1}{[\G:\G_{m}]}\sum_{\underline \s \in \G/\G_{m}}I_{m}(\m \circ \underline \s),
$$
and since $I_m$ is $\G/\G_m$-invariant, we have that
$$
I_{m}(\overline \m)\le aC(f)\frac{[\G:\G_{m}]}{t_{m}}.
$$
When we denote by $\P_{m, aC(f)}$ the set of probability measures $\m$ which is $\G/\G_{m}$-invariant with $I_{m}(\m) \le aC(f)[\G:\G_{m}]/t_{m}$, then
$$
\frac{1}{[\G:\G_{m}]}\l_{o, m, \e, i}(a) \le a \sup_{\m \in \P_{m, aC(f)}}\E_{\m}\left|\overline f_{o, i} - \langle f_{o}\rangle (\overline \y_{o, b(\e\sqrt{t_{m}})})\right|.
$$
Hence (\ref{eigenvalue}) follows from Theorem \ref{equilibrium}.
\end{proof}

Henceforth we often identify a probability measure on $Z_{m}$ with the one on $Z$ by the periodic extension:
Let $\pi_{m}:V \to V_{m}$ be the covering map induced by the $\G$-action.
Define a periodic inclusion 
$\i_{m}: Z_{m} \to Z$ by $(\i_{m}\y)_{z}:=\y_{\pi_{m}(z)}$, $\y \in Z_{m}, z \in V$.
We identify $\m$ on $Z_{m}$ with its push forward by $\i_{m}$, which is a periodic extension of $\m$ on $Z$.
Conversely,
we identify a $\G_{m}$-invariant probability measure on $Z$ with the one on $Z_{m}$ in a natural way.

Theorem \ref{equilibrium} follows the one-block estimate (Theorem \ref{1block}) and the two-blocks estimate (Theorem \ref{2blocks}).
First, we prove the one-block estimate.

\begin{thm}[The one-block estimate]\label{1block}
For every $\G$-invariant local function bundle $f:V\times Z \to \R$ and for every $C>0$, it holds that
$$
\lim_{i \to \infty}\limsup_{m \to \infty}\sup_{\m \in \P_{m,C}}\E_{\m}\left|\overline f_{o, i} - \langle f_{o}\rangle (\overline \y_{o,i})\right|=0.
$$
\end{thm}

We discuss a restricted region in $X$ and define the corresponding Dirichlet form.
Let $\L=(V_{\L}, E_{\L})$  be a subgraph of $X$.
Define the restricted configuration space by
$Z_{\L}:=\{0, 1\}^{V_{\L}}$,
and the $(1/2)$-Bernoulli measure $\n_{\L}$ on $Z_{\L}$
by the direct product of the $(1/2)$-Bernoulli measures on $\{0,1\}$.

In our setting, 
$\G$ acts on $X$ vertex transitively, thus $o \in V$ is a fundamental domain in $V$.
We can choose a fundamental domain $E^{0}$ in $E$ as the set of those edges; $oe$ or $te=o$.
We use the same notation $E^{0}$ for the image of $E^{0}$ on $X_{m}$ via the covering map.
Define the operator on $L^{2}(Z_{\L},\n_{\L})$ by
$$
L_{\L}^{\circ}:=\frac{1}{2}\sum_{e \in E_{\L}}\pi_{e}.
$$
For $\m \in \P(Z)$, we denote by $\m|_{\L}$ the restriction of $\m$ on $Z_{\L}$, that is, defined by taking the average outside of $\L$.
Set $\f_{\L}:=d\m|_{\L}/d\n_{\L}$ the density of $\m|_{\L}$.
The corresponding Dirichlet form of $\sqrt{\f_{\L}}$ is 
$$
I_{\L}^{\circ}(\m):=-\int_{Z_{\L}}\sqrt{\f_{\L}}L_{\L}^{\circ}\sqrt{\f_{\L}}d\n_{\L}.
$$

\begin{proof}[Proof of Theorem \ref{1block}]
Define a subgraph $\L$ of $X$ by setting 
$V_{\L}:=B(o, K)$ and $E_{\L}:=\{e \in E \ : \ oe, te \in V_{\L}\}$.
For every $\m \in \P_{m,C}$, by the convexity of the Dirichlet form and by the $(\G/\G_{m})$-invariance of $\m$ and $\n_{m}$, we have
\begin{align*}
I_{\L}^{\circ}(\m)	&=\frac{1}{2}\sum_{e \in E_{\L}}\int_{Z_{\L}}(\pi_{e}\sqrt{\f_{\L}})^{2}d\n_{\L} \\
				&\le \frac{1}{2}\sum_{e \in E_{\L}}\int_{Z_{m}}(\pi_{e}\sqrt{\f})^{2}d\n_{m} \\
				&\le \frac{1}{2}\sum_{\underline \s \in \G/\G_{m}, |\s|_{\G} \le K}\sum_{e \in \underline \s E^{0}}\int_{Z_{m}}(\pi_{e}\sqrt{\f})^{2}d\n_{m} \\
				&= \frac{1}{2}|B_{\G}(K)|\sum_{e \in E^{0}}\int_{Z_{m}}(\pi_{e}\sqrt{\f})^{2}d\n_{m},
\end{align*}
where in the last term $B_{\G}(K)$ denotes the ball of radius $K$ in $\G$ about $id$ in $\G$ by the word norm $|\cdot|_{\G}$.
On the other hand in the same way,
we have that
\begin{equation}\label{eq}
I_{m}(\m)=[\G:\G_{m}](1/2)\sum_{e \in E^{0}}\int_{Z_{m}}c(e, \y)(\pi_{e}\sqrt{\f})^{2}d\n_{m}.
\end{equation}
By the uniformly boundedness of $c(e,\y)\ge c_{0} >0$ (the non-degeneracy of $c(\cdot, \cdot)$), it holds that
$$
I_{\L}^{\circ}(\m)\le \frac{|B_{\G}(K)|}{c_{0}[\G:\G_{m}]}I_{m}(\m).
$$
Since $\m$ satisfies $I_{m}(\m)\le C[\G:\G_m]/t_{m}$,
we have that $I_{\L}^{\circ}(\m)\le C_{K}/ t_{m} \to 0$ as $m \to \infty$, where $C_{K}$ is a constant depending only on $K$.
The space of probability measures $\P(Z)$ is compact with the weak topology.
Thus, every sequence $\{\m_{i}\}_{i=1}^{\infty}$ in $\P(Z)$ has a convergence subsequence.
Let $\A \subset \P(Z)$ be the set of all limit points of $\{\m_{i}\}_{i=1}^{\infty}$ in $\P(Z)$. 
By the above argument, we have that $I_{\L}^{\circ}(\m)=0$ for every $\m \in \A$.
We obtain that $\m|_{\L}(\y^{e})=\m|_{\L}(\y)$ for all $e \in E_{\L}$ and all $\y \in Z_{\L}$ since 
$$
I_{\L}^{\circ}(\m)=(1/2)\sum_{e \in E_{\L}}\sum_{\y \in Z_{\L}}(\pi_{e}\sqrt{\m|_{\L}(\y)})^{2}=0.
$$
This holds for every $\L$ with radius $K$ and implies that random variables $\{\y_{x}\}_{x \in V}$ are exchangeable under $\m$.
By the de Finetti theorem, there exists a probability measure $\l$ on $[0,1]$ such that
$\m = \int_{0}^{1}\n_{\r}\l(d\r)$,
where $\n_{\r}$ is the $\r$-Bernoulli measure on $Z$.
Then we have 
\begin{align*}
\limsup_{m \to \infty}\sup_{\m \in \P_{m,C}}\E_{\m}\left|\overline f_{o,i}-\langle f_{o}\rangle(\overline \y_{o,i})\right|
&\le \sup_{\m \in \A}\E_{\m}\left|\overline f_{o,i}- \langle f_{o} \rangle(\overline \y_{o,i})\right| \\
&\le \sup_{\r \in [0,1]}\E_{\n_{\r}}\left|\overline f_{o, i}- \langle f_{o}\rangle(\overline \y_{o, i})\right|.
\end{align*}
Therefore, it is enough to show that 
$\lim_{i \to \infty}\sup_{\r \in [0,1]}\E_{\n_{\r}}\left|\overline f_{o, i} - \langle f_{o}\rangle(\overline \y_{o, i})\right|=0$.
Since $f$ is a $\G$-invariant local function bundle, there exists $L \ge 0$ such that 
$f_{o}:Z \to \R$ depends only on $\{\y_{z} \ : \ d(o,z) \le L\}$ and there exists a constant $C(f)>0$ depending only on $f$ such that
\begin{equation}\label{eqst}
\E_{\n_{\r}}\left|\overline f_{o,i} - \E_{\n_{\r}}[\overline f_{o,i}]\right|^{2} \le C(f)C_{L}/|F_{i}| \to 0
\end{equation}
as $i \to \infty$, where $C_{L}$ is a constant depending only on $L$.
Here we have that $\langle f_{o} \rangle(\r)=\E_{\n_{\r}}[\overline f_{o,i}]$ since $f$ is $\G$-invariant.
Since $\langle f_{o} \rangle(\r)$ is a polynomial with respect to $\r$, in particular, uniformly continuous on $[0,1]$,
it holds that
$$
\sup_{\r \in [0,1]}\E_{\n_{\r}}\left|\langle f_{o}\rangle(\overline \y_{o,i})- \langle f_{o}\rangle(\r)\right| \to 0, i\to \infty.
$$
By the triangular inequality,
\begin{align*}
&\sup_{\r \in [0,1]}\E_{\n_{\r}}\left|\overline f_{o,i}-\langle f_{o} \rangle (\overline \y_{o,i})\right| \\
&\le \sup_{\r \in [0,1]}\E_{\n_{\r}}\left|\overline f_{o,i} - \langle f_{o}\rangle(\r)\right|+
\sup_{\r \in [0,1]}\E_{\n_{\r}}\left|\langle f_{o}\rangle(\r)-\langle f_{o}\rangle(\overline \y_{o,i})\right| \to 0, i \to \infty.
\end{align*}
This concludes that $\lim_{i \to \infty}\limsup_{m \to \infty}\sup_{\m \in \P_{m,C}}\E_{\m}\left|\overline f_{o,i} - \langle f_{o}\rangle (\overline \y_{o,i})\right|=0$.
\end{proof}

Next, we prove the two-blocks estimate.

\begin{thm}[The two-blocks estimate]\label{2blocks}

For every $C >0$, it holds that
$$
\lim_{i \to \infty}\limsup_{\e \to 0}\limsup_{L \to \infty}\limsup_{m \to \infty}\sup_{\s \in \G \text{s.t.} \\ L < |\s| \le \e \sqrt{t_{m}}}\sup_{\m \in \P_{m,C}}\E_{\m}\left|\overline \y_{o, i} - \overline \y_{\s o, i}\right|=0.
$$
\end{thm}

We introduce the following notions:
Denote by $\P(Z \times Z)$ the space of probability measures on $Z \times Z$.
For $\s \in \G$, define $\hat \s: Z \to Z\times Z$ by
$\hat \s(\y):=(\y, \s^{-1}\y)$.
For $\m \in \P(Z)$, denote by $\hat \s \m \in \P(Z \times Z)$ the push forward by $\hat \s$, i.e., $\hat \s \m:=\m \circ \hat \s^{-1}$.
We define the subset $\A_{\e, L}$ in $\P(Z \times Z)$ as the set of all limit points of
$\{\hat \s \m \ : \ L < |\s| \le \e \sqrt{t_{m}}, \m \in \P_{m,C}\}$
as $m \to \infty$,
and the subset $\A_{\e}$ in $\P(Z \times Z)$ as the set of all limit points of $\A_{\e, L}$
as $L \to \infty$.

Then we have that:

$$
\limsup_{L \to \infty}\limsup_{m \to \infty}\sup_{\s \in \G \text{s.t.} \\ L < |\s| \le \e \sqrt{t_{m}}}\sup_{\m \in \P_{m,C}}\int_{(\y,\y') \in Z \times Z}\left|\overline \y_{o,i}- \overline \y'_{o,i}\right|(\hat \s \m)(d\y d\y') 
$$
$$
\le \sup_{\m \in \A_{\e}}\int_{(\y, \y') \in Z \times Z}\left|\overline \y_{o,i}- \overline \y'_{o,i}\right|\m(d\y d\y').
$$

As in the proof of the one block estimate, we introduce a subgraph $\L$ in $X$,
by setting $V_{\L}:=B(o, K)$, $E_{\L}:=\{e \in E \ : \ oe, te \in V_{\L}\}$
and consider the generator $L_{\L}^{\circ}$ on $L^{2}(Z, \n)$ by
$$
L_{\L}^{\circ}=\frac{1}{2}\sum_{e \in E_{\L}}\pi_{e}.
$$
Then let us define the two generators on $L^{2}(Z\times Z, \n\otimes \n)$ and the corresponding Dirichlet forms.
First, we define the generator on $L^{2}(Z\times Z, \n\otimes \n)$ by $L_{\L}^{\circ}\otimes1 + 1\otimes L_{\L}^{\circ}$.
For $\m \in \P(Z\times Z)$, denote by $\m|_{\L \times \L}$ the restriction of $\m$ on $Z_{\L} \times Z_{\L}$, i.e., 
$\m|_{\L\times \L}(\y, \y'):=\m(\{(\widetilde\y, \widetilde\y') \ : \ \widetilde\y|_{\L}=\y, \widetilde\y'|_{\L}=\y' \})$
for $(\y, \y') \in Z_{\L} \times Z_{\L}$.
The corresponding Dirichlet form of $\sqrt{\f_{\L\times \L}}$, where $\f_{\L\times \L}:=d\m|_{\L\times \L}/d\n_{\L}\otimes\n_{\L}$ is defined by
$$
I_{\L \times \L}^{\circ}(\m):= - \int_{Z_{\L}\times Z_{\L}}\sqrt{\f_{\L\times \L}}(L_{\L}^{\circ}\otimes 1+1\otimes L_{\L}^{\circ})\sqrt{\f_{\L\times \L}}d\n_{\L} \otimes \n_{\L}.
$$
Second, we define the generator on $L^{2}(Z\times Z, \n\otimes \n)$ in the following way:
For $(x, y) \in V\times V$ and $(\y, \y') \in Z \times Z$, we construct a new configuration $(\y, \y')^{(x, y)} \in Z \times Z$ by setting $(\y'_{y}, \y_{x})$ on $(x, y)$, $(\y'_{x}, \y_{y})$ on $(y, x)$ and keeping unchanged otherwise.
For $F \in L^{2}(Z\times Z, \n\otimes \n)$, $(x, y) \in V \times V$,
define $\widetilde{\pi}_{x, y}F((\y, \y'))=F((\y, \y')^{(x, y)})-F((\y, \y'))$.
Define the generator $L_{o,o}^{\circ}$ on $L^{2}(Z\times Z, \n\otimes \n)$ by 
$$
L_{o, o}^{\circ}:=\widetilde{\pi}_{o, o},
$$
and the corresponding Dirichlet form of $\sqrt{\f_{\L\times \L}}$ by
$$
I_{\L \times \L}^{(o,o)}(\m):= - \int_{Z_{\L}\times Z_{\L}}\sqrt{\f_{\L\times \L}}L_{o,o}^{\circ}\sqrt{\f_{\L\times \L}}d\n_{\L}\otimes \n_{\L}.
$$

Then we use the following lemma. The proof appears in Lemma 4.2 in \cite{T}, so we omit it.

\begin{lem}\label{pathlemma}
For all $F \in L^{2}(Z_{m}, \n_{m})$ and all $\underline \s \in \G/\G_{m}$, it holds that
$$
\int_{Z_{m}}(\pi_{o,\underline \s o}F)^{2}d\n_{m} \le 4 d(o, \underline \s o)^{2}\sum_{e \in E^{0}}\int_{Z_{m}}(\pi_{e}F)^{2}d\n_{m}.
$$
\end{lem}

For any constant $\widetilde C >0$,
let
$$
\A_{\e, \widetilde C}:=\{\m \in \P(Z \times Z) \ : \ I_{\L\times \L}^{\circ}(\m)=0, I_{\L\times \L}^{(o,o)}(\m) \le \widetilde C \e^{2}\}.
$$
Then we have the following lemma.

\begin{lem}\label{inclusion}
There exists a constant $\widetilde C>0$ such that $\A_{\e} \subset A_{\e, \widetilde C}$.
\end{lem}

\begin{proof}
Recall that a subgraph $\L=(V_{\L}, E_{\L})$ of $X$ is defined by
$V_{\L}:=B(o, K)$ and $E_{\L}:=\{e \in E \ : \ oe, te \in V_{\L}\}$.
Here we regard $X_{m}$ as a suitable fundamental domain in $X$ by the $\G_{m}$-action.
Take large enough $L$, and $m$ for the diameter of $\L$ so that for all $\s$ with $L < |\s|$, $V_{\L}\cap \s V_{\L} = \emptyset$,
for all $\s$ with $L < |\s| \le \e \sqrt{t_{m}}$,
$V_{\L}\cap \s V_{\L} \subset X_{m}$.
Notice that $\sqrt{t_{m}} \le 2diamX_m$.

For given $(\y, \y') \in Z_{\L}\times Z_{\L}$,
we define $\widetilde \y \in Z_{\L \cup \s \L}$ by $\widetilde \y|_{\L}=\y, \s^{-1}\widetilde \y|_{\L}=\y'$.
Then $(\hat \s \m)|_{\L \times \L}(\y, \y')=\m|_{\L \cup \s \L}(\widetilde \y)$.
By using the generator $L_{\L\cup \s \L}^{\circ}$ on $L^{2}(Z, \n)$,
we have that
$I_{\L \times \L}^{\circ}(\hat \s \m)=I_{\L \cup \s \L}^{\circ}(\m)$.
For every $\m \in \P_{m, C}$ as in the previous section by using (\ref{eq}), and by the convexity of the Dirichlet form and the $(\G/\G_{m})$-invariance of $\m$ and $\n_{m}$,
we have that
$$
I_{\L\times \L}^{\circ}(\hat \s \m)=I_{\L \cup \s \L}^{\circ}(\m) \le \frac{|B_{\G}(K)|}{2c_{o}[\G:\G_{m}]}I_{m}(\m)\le C\frac{|B_{\G}(K)|}{2c_{0}t_{m}} \to 0,
$$
as $m \to \infty$.
Therefore $I_{\L \times \L}^{\circ}(\m_{L})=0$ for all $\m_{L} \in \A_{\e, L}$, and $I_{\L\times \L}^{\circ}(\widetilde \m)=0$ 
for all $\widetilde \m \in \A_{\e}$
by the continuity of $I_{\L \times \L}^{\circ}$.
Furthermore, by the convexity of the Dirichlet form,
$$
I_{\L\times \L}^{(o,o)}(\hat \s \m)=\frac{1}{2}\int_{Z_{\L\cup\s \L}}\left(\pi_{o,\underline \s o}\sqrt{\frac{d\m|_{\L\cup\s \L}}{d\n_{\L \cup \s \L}}}\right)^{2}d\n_{\L\cup\s \L} \le \frac{1}{2}\int_{Z_{m}}\left(\pi_{o,\underline \s o}\sqrt{\frac{d\m}{d\n_{m}}}\right)^{2}d\n_{m}.
$$
By Lemma \ref{pathlemma}, the $(\G/\G_{m})$-invariance of $\m \in \P_{m, C}$ and (\ref{eq}),
\begin{align*}
\int_{Z_{m}}\left(\pi_{o, \underline \s o}\sqrt{\frac{d\m}{d\n_{m}}}\right)^{2}d\n_{m} &\le
4d(o, \underline \s o)^{2}\sum_{e \in E^{0}}\int_{Z_{m}}\left(\pi_{e}\sqrt{\frac{d\m}{d\n_{m}}}\right)^{2}d\n_{m} \\
&\le 4d(o,\underline \s o)^{2}\frac{2}{c_{0}[\G:\G_{m}]}I_{m}(\m).
\end{align*}

Since $d(o, \underline \s o)\le |\s|_{\G}$, we have that for $\s \in \G$ with $L < |\s| \le \e \sqrt{t_{m}}$
$$
I_{\L \times \L}^{(o,o)}(\hat \s \m)\le \frac{1}{2}\cdot 4(\e \sqrt{t_{m}})^{2}\frac{2}{[\G:\G_{m}]}\frac{C}{c_{0}}\frac{[\G:\G_{m}]}{t_{m}} = \frac{4C}{c_{0}}\e^{2}.
$$
Define $\widetilde C:=4C/c_{0}$. 
By the continuity of $I_{\L \times \L}^{(o,o)}$, we have that $I_{\L \times \L}^{(o, o)}(\widetilde \m) \le \widetilde C \e^{2}$ for all $\widetilde \m \in \A_{\e}$.
This implies that $\A_{\e} \subset \A_{\e, \widetilde C}$.
\end{proof}

\begin{proof}[Proof of Theorem \ref{2blocks}]
Let $\A_{0}$ be the set of all limit points of $\A_{\e, \widetilde C}$ as $\e \to 0$.
It holds that
$I_{\L \times \L}^{\circ}(\widetilde \m_{0})=0$
and $I_{\L \times \L}^{(o, o)}(\widetilde \m_{0})=0$ for all $\widetilde \m_{0} \in \A_{0}$ on each $\L$.
This shows that the random variables $\{(\y_{x}, \y'_{y})\}_{x, y \in V}$ are exchangeable on $Z \times Z$ under $\tilde \m_{0}$.
By the de Finetti theorem,
there exists a probability measure $\l$ on $[0,1]$ such that
$$
\widetilde \m_{0}=\int_{[0,1]}\n_{\r}\otimes \n_{\r}\l(d\r).
$$
In the proof of Theorem \ref{1block} (\ref{eqst}), we have
$$
\lim_{i \to \infty}\sup_{\r \in [0,1]}\E_{\n_{\r}}\left|\overline \y_{o, i} - \r\right|^{2}=0.
$$
Then 
$$
\sup_{\widetilde \m_{0} \in \A_{0}}\E_{\widetilde \m_{0}} \left|\overline \y_{o, i} - \overline \y'_{o, i}\right|
\le \sup_{\r \in [0,1]}\E_{\n_{\r}\otimes \n_{\r}}\left|\overline \y_{o, i} - \overline \y'_{o, i}\right|
\le 2 \sup_{\r \in [0,1]}\E_{\n_{\r}}\left|\overline \y_{o, i} - \r \right| \to 0,
$$
as $i \to \infty$.
Here we used the triangular inequality in the last inequality.
By Lemma \ref{inclusion}, we have that $\A_{\e} \subset \A_{\e, \widetilde C}$ for some $\widetilde C >0$,
\begin{align*}
&\limsup_{\e \to 0}\limsup_{L \to \infty}\limsup_{m \to \infty}\sup_{\s \in \G \text{s.t.} L < |\s| \le \e\sqrt{t_{m}}}\sup_{\m \in \P_{m, C}}\E_{\m}\left|\overline \y_{o, i} - \overline \y_{\underline \s o, i}\right|  \\
&\le \limsup_{\e \to 0}\sup_{\widetilde \m \in \A_{\e}}\E_{\widetilde \m}\left|\overline \y_{o, i} - \overline \y'_{o, i}\right| \\
&\le \limsup_{\e \to 0}\sup_{\widetilde \m \in \A_{\e, \widetilde C}}\E_{\widetilde \m}\left|\overline \y_{o, i} - \overline \y'_{o, i}\right|	\\
&\le \sup_{\widetilde \m_{0} \in \A_{0}}\E_{\widetilde \m_{0}}\left|\overline \y_{o, i} - \overline \y'_{o, i}\right| \to 0,
\end{align*}
as $i \to \infty$.
This proves the theorem.
\end{proof}

\begin{proof}[Proof of Theorem \ref{equilibrium}]
For $\y \in Z$, the following uniform estimates hold:
\begin{equation}\label{eq1}
\left| \overline \y_{o, b(\e \sqrt{t_{m}})} - \frac{1}{\left| F_{b(\e \sqrt{t_{m}})}\right|}\sum_{\s \in F_{b(\e\sqrt{t_{m}})}}\overline \y_{\s o, i} \right|
\le C(F_i)\frac{\left|\partial_SF_{b(\e \sqrt{t_{m}})}\right|}{\left|F_{b(\e \sqrt{t_{m}})}\right|},
\end{equation}
where $C(F_i)$ is a constant depending only on $F_i$, and

\begin{equation}\label{eq2}
\left|\frac{1}{\left| F_{b(\e \sqrt{t_{m}})}\right|}\sum_{\s \in F_{b(\e\sqrt{t_{m}})}}\overline \y_{\s o, i} -\frac{1}{\left| F_{b(\e \sqrt{t_{m}})}\right|}\sum_{\s \in F_{b(\e\sqrt{t_{m}})}\setminus B_{\G}(L)}\overline \y_{\s o, i} \right| 
\le \frac{\left| B_{\G}(L)\right|}{\left| F_{b(\e \sqrt{t_{m}})}\right|}.
\end{equation}
By the triangular inequality, (\ref{eq1}) and (\ref{eq2}), we have that

\begin{align*}
&\left| \overline \y_{o, b(\e \sqrt{t_{m}})} - \frac{1}{\left| F_{b(\e \sqrt{t_{m}})}\setminus B_{\G}(L)\right|}
\sum_{\s \in F_{b(\e\sqrt{t_{m}})}\setminus B_{\G}(L)}\overline \y_{\s o, i} \right|  
\le C(F_i)\frac{\left|\partial_SF_{b(\e \sqrt{t_{m}})}\right|}{\left|F_{b(\e \sqrt{t_{m}})}\right|}
+2\frac{\left| B_{\G}(L)\right|}{\left| F_{b(\e \sqrt{t_{m}})}\right|},
\end{align*}
where 
the last terms go to $0$
as $m \to \infty$,
since a F{\o}lner sequence satisfies that $|\partial_{S} F_{b(\e\sqrt{t_{m}})}|/|F_{b(\e\sqrt{t_{m}})}| \to 0$ as $m \to \infty$.

For every $\m \in \P_{m,C}$, we have that
$$
\E_{\m}\left|\overline \y_{o,i} - \overline \y_{o, b(\e\sqrt{t_{m}})}\right|
\le \E_{\m}\left| \overline \y_{o, i} - \frac{1}{\left| F_{b(\e \sqrt{t_{m}})}\setminus B_{\G}(L)\right|}
\sum_{\s \in F_{b(\e\sqrt{t_{m}})}\setminus B_{\G}(L)}\overline \y_{\s o, i} \right|  + o_{m},
$$
where $o_m \to 0$ as $m \to \infty$.
The last term is bounded by
\begin{align*}
\frac{1}{\left| F_{b(\e \sqrt{t_{m}})}\setminus B_{\G}(L)\right|}\sum_{\s \in F_{b(\e \sqrt{t_{m}})}\setminus B_{\G}(L)}\E_{\m}\left|\overline \y_{o,i} - \overline \y_{\s o, i}\right| 
&\le \sup_{\s \in F_{b(\e \sqrt{t_{m}})}\setminus B_{\G}(L)}\E_{\m}\left|\overline \y_{o,i} - \overline \y_{\s o, i}\right|  \\
&\le \sup_{\s \in \G \text{s.t.} L < |\s|_{\G} \le \e \sqrt{t_{m}}}\sup_{\m \in \P_{m,C}}\E_{\m}\left|\overline \y_{o,i} - \overline \y_{\s o, i}\right|.
\end{align*}
By Theorem \ref{2blocks}, it holds that
$$
\lim_{i \to \infty}\limsup_{\e \to 0}\limsup_{m \to \infty}\sup_{\m \in \P_{m, C}}\E_{\m}\left|\overline \y_{o,i} - \overline \y_{o, b(\e\sqrt{t_{m}})}\right|=0.
$$
Since for every $\G$-periodic local function bundle $f: V \times Z \to \R$,
the function $\langle f_{o}\rangle(\cdot)$ is uniformly continuous on $[0,1]$,
$$
\lim_{i \to \infty}\limsup_{\e \to 0}\limsup_{m \to \infty}\sup_{\m \in \P_{m, C}}\E_{\m}\left|\langle f_{o}\rangle (\overline \y_{o,i}) - \langle f_{o}\rangle (\overline \y_{o, b(\e\sqrt{t_{m}})})\right|=0.
$$
By Theorem \ref{1block} and the triangular inequality, we conclude that
$$
\lim_{i \to \infty}\limsup_{\e \to 0}\limsup_{m \to \infty}\sup_{\m \in \P_{m, C}}\E_{\m}\left|\overline f_{o,i} - \langle f_{o}\rangle (\overline \y_{o, b(\e\sqrt{t_{m}})})\right|=0,
$$
and this completes the proof.
\end{proof}

\textbf{Acknowledgements.}
The author thanks Hiroshi Kawabi for helpful comments and encouragement, Motoko Kotani for support during this work and an anonymous referee for useful comments to improve this note.

\end{document}